\documentclass[11pt]{epiarticle}
\usepackage{epimath}

\usepackage{graphicx}
\usepackage{amsfonts}
\usepackage{commath}
\usepackage{graphics}
\usepackage{hyperref}
\usepackage{tikz}
\usetikzlibrary{arrows}

\setpapertype{A4}

 \hypersetup{linkcolor=blue,citecolor=red,filecolor=green,urlcolor=magenta} 

\usepackage[english]{babel}
 
\newtheorem{definition}{Definition}[section]

\title{Godement--Jacquet $L$-function, some conjectures and some consequences}
\titlemark{Godement--Jacquet $L$-function, some conjectures and some consequences}
\author{Amrinder Kaur \&\ Ayyadurai Sankaranarayanan}
\authormark{A.\ Kaur \&\ A.\ Sankaranarayanan}
\date{} 
\journal{Hardy-Ramanujan Journal -- (yyyy), ---} 
\acceptation{submitted dd/mm/yyyy, accepted dd/mm/yyyy, revised dd/mm/yyyy}

\begin{document}

\maketitle

\thanks{We thank \href{http://episciences.org}{episciences.org} for providing open access hosting of the electronic journal \emph{Hardy-Ramanujan Journal}}

\begin{prelims}

\def\abstractname{Abstract}
\abstract{In this paper, we investigate the mean square estimate for the logarithmic derivative of the Godement--Jacquet $L$-function $L_f(s)$ assuming the Riemann hypothesis for $L_f(s)$ and Rudnick--Sarnak conjecture.}



\end{prelims}

\section{Introduction} 

\noindent
Let $n \geq 2$, and let $v=(v_1,v_2,\dots,v_{n-1}) \in \mathbb{C}^{n-1}$. A Maass form \cite{Dg} for $SL(n,\mathbb{Z})$ of type $v$ is a smooth function $f \in \mathcal{L}^2(SL(n,\mathbb{Z}) \backslash \mathcal{H}^n)$ which satisfies
\begin{enumerate}
\item $f(\gamma z) = f(z)$, for all $\gamma \in SL(n,\mathbb{Z}), z \in \mathcal{H}^n$,
\item $Df(z) = \lambda_D f(z)$, for all $D \in \mathfrak{D}^n$ where $\mathfrak{D^n}$ is the center of the universal enveloping algebra of $\mathfrak{gl}(n,\mathbb{R})$ and $\mathfrak{gl}(n,\mathbb{R})$ is the Lie algebra of GL(n,$\mathbb{R}$).
\item $\int \limits_{(SL(n,\mathbb{Z}) \cap U) \backslash U} f(uz) \ du = 0, $\\
for all upper triangular groups $U$ of the form 
$$U = \left\{
\begin{pmatrix} 
I_{r_1} &             &           &  \\
           & I_{r_2}  &           & *\\ 
           &             & \ddots &   \\ 
           &             &           & I_{r_b}  \\ 
\end{pmatrix} \right\}, $$
with $r_1+r_2+\cdots+r_b=n$. Here, $I_r$ denotes the $r \times r$ identity matrix, and $*$ denotes arbitrary real entries. \\

\end{enumerate}

\noindent
A Hecke--Maass form is a Maass form which is an eigenvector for the Hecke operators algebra. 

\noindent
Let $f(z)$ be a Hecke--Maass form of type $v=(v_1,v_2,\dots,v_{n-1}) \in \mathbb{C}^{n-1}$ for $SL(n,\mathbb{Z})$. Then it has the Fourier expansion 
\begin{align*}
f(z) &= \sum_{\gamma \in U_{n-1}(\mathbb{Z}) \backslash SL(n-1,\mathbb{Z})} \sum_{m_1=1}^{\infty} \dots \sum_{m_{n-2}=1}^{\infty} \sum_{m_{n-1} \neq 0} \frac{A(m_1,\dots,m_{n-1})}{\prod_{j=1}^{n-1} \abs{m_j}^\frac{j(n-j)}{2}} \\
& \quad \times W_J \left( M \cdot \begin{pmatrix} \gamma &  \\  & 1 \end{pmatrix} z, v, \psi_{1,\dots,1,\frac{m_{n-1}}{\abs{m_{n-1}}}}\right) ,
\end{align*}

\noindent
where 
$$ M = 
\begin{pmatrix} 
 m_1\dots m_{n-2} \cdot \abs{m_{n-1}} &           &             &        & \\
                                                        & \ddots &              &        & \\ 
                                                        &           & m_1m_2 &        & \\ 
                                                        &           &             & m_1 & \\ 
                                                        &           &             &        & 1 
\end{pmatrix} ,
$$

$$ A(m_1,\dots,m_{n-1}) \in \mathbb{C} , \qquad  A(1,\dots,1)=1,$$

$$ \psi_{1,\dots,1,\epsilon} \left( 
\begin{pmatrix}
1 & u_{n-1} &             &            &   \\
  & 1          & u_{n-2}  &            & * \\
  &            & \ddots     & \ddots  &     \\
  &            &              & 1          & u_1 \\
  &            &              &             & 1        
\end{pmatrix}
\right) = e^{2 \pi i (u_1+\cdots +u_{n-2}+\epsilon u_{n-1})},$$

\noindent
$U_{n-1}(\mathbb{Z})$ denotes the group of $(n-1) \times (n-1)$ upper triangular matrices with $1s$ on the diagonal and an integer entry above the diagonal and $W_J$ is the Jacquet Whittaker function. \\

\noindent
If $f(z)$ is a Maass form of type $(v_1,\dots,v_{n-1}) \in \mathbb{C}^{n-1}$, then
$$ \tilde{f}(z) := f(w \cdot (z^{-1})^T \cdot w), $$
$$ w = \begin{pmatrix} 
&&& (-1)^{ \left[ \frac{n}{2} \right]} \\
&& 1 & \\
& \reflectbox{$\ddots$} && \\
1&&& 
 \end{pmatrix} $$
is a Maass form of type $(v_{n-1},\dots,v_1)$ for $SL(n,\mathbb{Z})$ called the \emph{dual Maass form}. If $A(m_1, \dots, m_{n-1})$ is the $(m_1, \dots, m_{n-1})$--Fourier coefficient of $f$, then $A(m_{n-1} , \dots, m_1)$ is the corresponding Fourier coefficient of $\tilde{f}$. \\

\noindent
We note that the Fourier coefficients $A(m_1,\dots,m_{n-1})$ satisfy the multiplicative relations
$$ A(m_1m_1',\dots,m_{n-1}m_{n-1}') = A(m_1,\dots,m_{n-1}) \cdot A(m_1',\dots,m_{n-1}'), $$
if
$$ (m_1 \dots m_{n-1},m_1' \dots m_{n-1}') = 1 ,$$

$$ A(m,1,\dots,1)A(m_1,\dots,m_{n-1}) = \sum_{\prod \limits_{l=1}^n c_l=m \atop c_1|m_1,c_2|m_2,\dots,c_{n-1}|m_{n-1}} A \left(  \frac{m_1c_n}{c_1},\frac{m_2c_1}{c_2},\dots,\frac{m_{n-1}c_{n-2}}{c_{n-1}} \right), $$
and
$$ A(m_{n-1},\dots,m_1) = \overline{A(m_1,\dots,m_{n-1})}. $$ \\

\begin{definition} \cite{Dg}
The Godement--Jacquet $L$-function $L_f(s)$ attached to $f$ is defined for $\Re(s) >1$ by 
$$ L_f(s) = \sum_{m=1}^{\infty} \frac{A(m,1,\dots,1)}{m^s} = \prod_p \prod_{i=1}^n (1-\alpha_{p,i}p^{-s})^{-1} ,$$ 
where $\{ \alpha_{p,i} \}, 1\leq i \leq n$ are the complex roots of the monic polynomial 

$$ X^n + \sum_{j=1}^{n-1} (-1)^j A(\overbrace{1,\dots,1}^{j-1 \; \text{terms}},p,1,\dots,1) X^{n-j} +(-1)^n \in \mathbb{C}[X], \quad \text{and}$$

$$ A(\overbrace{1,\dots,1}^{j-1},p,1,\dots,1) = \sum_{1 \leq i_1 < \dots < i_j \leq n} \alpha_{p,i_1} \dots \alpha_{p,i_j}, \qquad \text{for} \; \; 1 \leq j \leq n-1  .$$ \\
\end{definition}

\noindent
$L_f(s)$ satisfies the functional equation:
\begin{align*}
\Lambda_f(s) &:= \prod_{i=1}^n \pi^{\frac{-s+\lambda_i(v_f)}{2}} \Gamma \left( \frac{s-\lambda_i(v_f)}{2} \right) L_f(s) \\
&= \Lambda_{\tilde{f}}(1-s), 
\end{align*}
where $\tilde{f}$ is the Dual Maass form. \\

\noindent
In the case of Godement--Jacquet $L$-function, Yujiao Jiang and Guangshi L\"{u} \cite{YjGl} have studied cancellation on the exponential sum $\sum \limits_{m \leq N} \mu(m)A(m,1)e^{2 \pi i m \theta}$ related to $SL(3,\mathbb{Z})$ where $\theta \in \mathbb{R}$ . \\

\noindent
\emph{Throughout the paper, we assume that $f$ is self dual i.e., $\tilde{f}=f$.} \\ 
\emph{$\epsilon$, $\epsilon_1$ and $\eta$ always denote any small positive constants.} \\ \\
If $N_f(T)$ denotes the number of zeros of $L_f(s)$ in the rectangle mentioned below, then from the functional equation and the argument principle of complex function theory we have, 
$$N_f(T) \sim c(n) T \log T ,$$
where $c(n)$ is a non zero constant depending only on the degree $n$ of $L_f(s)$. \\

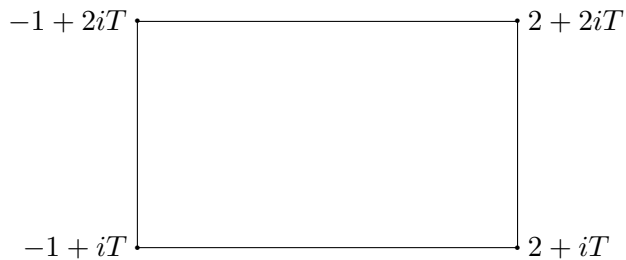
\begin{center}
\begin{tikzpicture}

\draw (0,3) -- (0,0) ; 
\draw (0,0) -- (5,0) ; 
\draw (5,0) -- (5,3) ; 
\draw (5,3) -- (0,3) ; 

\node at (0,0) {$\boldsymbol{\cdot}$};
\node at (0,3) {$\boldsymbol{\cdot}$};
\node at (5,0) {$\boldsymbol{\cdot}$};
\node at (5,3) {$\boldsymbol{\cdot}$};

\coordinate [label= left:$-1+iT$] (A) at (0,0);
\coordinate [label= left:$-1+2iT$] (D) at (0,3);
\coordinate [label= right:$2+iT$] (B) at (5,0);
\coordinate [label= right:$2+2iT$] (C) at (5,3); 

\end{tikzpicture} 
\end{center} 
\hfill\\

\noindent
\textbf{(i) The generalized Ramanujan conjecture:} \\
It asserts that
$$ \abs{A(m,1,\dots,1)} \leq d_n(m) $$
where $d_n(m)$ is the number of representations of $m$ as the product of $n$ natural numbers. The current best estimates are due to Kim and Sarnak \cite{Hk} for $2 \leq n \leq 4$ and Luo, Rudnick and Sarnak for $n \geq 5$
\begin{align*}
\abs{A(m)} \leq m^{\frac{7}{64}} d(m) ,\\
\abs{A(m,1)} \leq m^{\frac{5}{14}} d_3(m) ,\\
\abs{A(m,1,1)} \leq m^{\frac{9}{22}} d_4(m) ,\\
\abs{A(m,1,\dots,1)} \leq m^{\frac{1}{2}-\frac{1}{n^2+1}} d_n(m). \\
\end{align*}

\noindent
We note that the generalized Ramanujan conjecture is equivalent to 
$$ \abs{\alpha_{p,i}} =1 \qquad \forall \; \text{primes $p$ and} \; i=1,2,\dots,n.$$
Other estimates are equivalent to 
$$ \abs{\alpha_{p,i}} \leq p^{\theta_n} \qquad \forall \; \text{primes $p$ and} \; i=1,2,\dots,n \; \text{where}$$
\begin{align*}
\theta_2 := \frac{7}{64}, \qquad \theta_3 := \frac{5}{14}, \qquad \theta_4 := \frac{9}{22}, \qquad \theta_n := \frac{1}{2} - \frac{1}{n^2+1} (n \geq 5).\\
\end{align*}

\noindent
\textbf{(ii) Ramanujan's generalized weak conjecture:} \\
We formulate this conjecture as: \\
For $n \geq 2$, the inequality 
$$ \abs{\alpha_{p,i}} \leq p^{\frac{1}{4} - \epsilon_1} $$
holds for some small $\epsilon_1 >0 $, for every prime $p$ and for $i=1,2,\dots,n$. Of course, this weak conjecture holds good for $n=2$. For $n \geq 3$, this conjecture is still open. \\

\noindent
Taking the logarithmic derivative of $L_f(s)$, we have
$$ -{\frac{L_f'}{L_f}}(s) := \sum_{m=1}^{\infty} \frac{\Lambda_f(m)}{m^s} = \sum_{m=1}^{\infty} \frac{\Lambda(m) a_f(m)}{m^s} $$
where $a_f(m)$ is multiplicative and
$$ a_f(p^r) = \sum_{i=1}^n \alpha_{p,i}^r $$
for any integer $r \geq 1$. \\
In particular,
$$ a_f(p) = \sum_{i=1}^n \alpha_{p,i} = A(p,1,\dots,1) .$$ \\

\noindent
\textbf{(iii) Rudnick--Sarnak conjecture:} \\
For any fixed integer $r \geq 2$,
$$ \sum_p \frac{\abs{a_f(p^r)}^2 (\log p)^2}{p^r} < \infty .$$
We know that this conjecture is true for $n \leq 4$. (See \cite{Hh,ZrPs}.) \\

\noindent
\textbf{(iv) Riemann hypothesis for $L_f(s)$:} \\
It asserts that $L_f(s) \neq 0$ in $\Re(s) > \frac{1}{2}$. \\ \\

\noindent
The aim of this paper is to establish: \\
\begin{theorem} \label{t1}
Ramanujan's weak conjecture implies Rudnick--Sarnak conjecture. \\
\end{theorem}

\begin{remark}
\normalfont
Theorem \ref{t1} is indicated in \cite{Hh}.
\end{remark}

\begin{theorem} \label{t2}
Assume $n \geq 5$ be any arbitrary but fixed integer. Let $\epsilon$ be any small positive constant and $T \geq T_0$ where $T_0$ is sufficiently large. Assume the Rudnick--Sarnak conjecture and Riemann hypothesis for $L_f(s)$. Then the estimate:
$$ \int_T^{2T} \abs{\frac{L_f'}{L_f} \left( \sigma_0 + it \right) }^2 dt \ll_{f,n,\epsilon,\eta} T (\log T)^{2\eta}$$
holds for $\frac{1}{2}+\epsilon \leq \sigma_0 \leq 1-\epsilon$ with $\eta$ being some constant satisfying $0<\eta<\frac{1}{2}$. \\
\end{theorem}

\begin{remark}
\normalfont{
Since Rudnick--Sarnak conjecture is true for $2 \leq n \leq 4$, Theorem \ref{t2} holds just with the assumption of Riemann hypothesis for $L_f(s)$ whenever $2 \leq n \leq 4$. } \\
\end{remark}

\begin{remark}
\normalfont
It is not difficult to see from our arguments that only assuming Riemann Hypothesis for $L_f(s)$, Theorem \ref{t2} can be upheld for any $\sigma_0$ satisfying $1-\frac{1}{n^2+1}+ \epsilon \leq \sigma_0 \leq 1-\epsilon$ by using the bound $\theta_n = \frac{1}{2} - \frac{1}{n^2+1}$ of Luo, Rudnick and Sarnak. \\
It is also not difficult to see from our arguments that the generalized Ramanujan conjecture and the Riemann hypothesis for $L_f(s)$ together imply the bound 
\begin{equation} \label{e1}
\int_T^{2T} \abs{\frac{L_f'}{L_f}(\sigma_0+it)}^2 dt \ll_{f,n,\epsilon} T 
\end{equation}
to hold for any $\sigma_0$ satisfying $\frac{1}{2}+\epsilon \leq \sigma_0 \leq 1-\epsilon$. \\
Though we expect the bound stated in Equation \ref{e1} to hold unconditionally for $\sigma_0$ in the said range, this seems to be very hard. \\
\end{remark}

\section{Some Lemmas}

\begin{lemma} \label{l1}
If $f(s)$ is regular and 
$$ \abs{\frac{f(s)}{f(s_0)}} < e^M \qquad (M >1) $$
in \, $\abs{s-s_0} \leq r_1$, then for any constant $b$ with $0<b<\frac{1}{2}$,
$$ \abs{\frac{f'}{f}(s)-\sum_{\rho} \frac{1}{s-\rho}} \ll_b \frac{M}{r_1} $$
in \, $\abs{s-s_0} \leq \left( \frac{1}{2}-b \right) r_1$, where $\rho$ runs over all zeros of $f(s)$ such that $\abs{\rho - s_0} \leq \frac{r_1}{2}$. \\
\end{lemma}

\begin{proof}
See Lemma $\alpha$ in Section 3.9 of \cite{EtHb} or see \cite{KrAs}. \\
\end{proof}

\begin{lemma} \label{l2}
Let $N_f^*(T)$ denote the number of zeros of $L_f(s)$ in the region $0 \leq \sigma \leq 1$, $0 \leq t \leq T$. Then,
$$ N_f^*(T+1)-N_f^*(T) \ll_n \log T .$$ \\
\end{lemma}

\begin{proof}
Let $n(r_1)$ denote the number of zeros of $L_f(s)$ in the circle with centre $2+iT$ and radius $r_1$.
By Jensen's theorem,
$$ \int_0^3 \frac{n(r_1)}{r_1} dr_1 = \frac{1}{2 \pi} \int_0^{2 \pi} \log \abs{L_f \left( 2+iT+3e^{i \theta} \right)} d \theta - \log \abs{L_f (2+iT)} .$$
From the functional equation, we observe that
$$ \abs{L_f(s)} <_f t^A \qquad \text{for} \; -1 \leq \sigma \leq 5 \; \; \text{where $A$ is some fixed positive constant,}   $$
and hence we have,
$$\log \abs{L \left( 2+iT+3e^{i \theta} \right)} < A \log T .$$ \\

\noindent
Note that
\begin{align*}
\abs{1-\frac{\alpha_{p,i}}{p^{2+it}}} & \geq 1 - \frac{\abs{\alpha_{p,i}}}{p^2} \\
& \geq 1 - \frac{p^{\frac{1}{2}}}{p^2} \\
& = 1 - \frac{1}{p^{\frac{3}{2}}}. \\
\end{align*}

\noindent
Thus we have,
\begin{align*}
\abs{L_f(2+it)} &= \prod_p \prod_{i=1}^n \abs{ \left( 1-\frac{\alpha_{p,i}}{p^{2+it}} \right) }^{-1} \\
& \leq \prod_p \prod_{i=1}^n \left( 1-\frac{1}{p^{\frac{3}{2}}} \right)^{-1} \\
& \leq \left( \zeta \left( \frac{3}{2} \right) \right)^n \\
& \ll_n 1. \\
\end{align*}

\noindent
Therefore,
\begin{align*}
\int_0^3 \frac{n(r_1)}{r_1} dr_1 &< A \log T +A \ll \log T , \\
\int_0^3 \frac{n(r_1)}{r_1} dr_1 &\geq \int_{\sqrt{5}}^3 \frac{n(r_1)}{r_1} dr_1 \geq n(\sqrt{5}) \int_{\sqrt{5}}^3 \frac{dr_1}{r_1} \geq c.n(\sqrt{5}) . \\
\end{align*} 
Hence,
$$ N_f^*(T+1)-N_f^*(T) \ll_n \log T. $$ \\
\end{proof}

\begin{lemma} \label{l3}
Let $a_m$(m=1,2,\dots,N) be any set of complex numbers. Then
$$ \int_T^{2T} \abs{\sum_{m=1}^N a_m m^{-it} }^2 dt = \sum_{m=1}^N \abs{a_m}^2 \left( T+O(m) \right). $$ \\
\end{lemma}

\begin{lemma} \label{l4} 
Let $b_m$ be any set of complex numbers such that $\sum m \left( \, \abs{b_m} \right)^2$ is convergent. Then
$$ \int_T^{2T} \abs{\sum_{m=1}^{\infty} b_m m^{-it} }^2 dt = \sum_{m=1}^{\infty} \abs{b_m}^2 \left( T+O(m) \right). $$ \\
\end{lemma}

\begin{proof}
See \cite{HlRc} or \cite{Kr} for Montgomery and Vaughan theorem. \\
\end{proof}

\noindent
\emph{Hereafter, $Y \geq 10$ is an arbitrary parameter depending on $T$ which will be chosen suitably later. Also, $\sigma_0$ satisfies the inequality $\frac{1}{2}+\epsilon \leq \sigma_0 \leq 1-\epsilon$ for any small positive constant $\epsilon$.} \\

\begin{lemma} \label{l5}
For $\frac{1}{2}+\epsilon \leq \sigma_0 \leq 1- \epsilon$, we have
$$ \sum_{m> \frac{Y}{2}(\log Y)^2} \frac{m\abs{\Lambda_f(m)}^2 e^{-\frac{2m}{Y}}}{m^{2\sigma_0}} \ll 1 .$$ \\
\end{lemma}

\begin{proof}
We have,
\begin{align*}
\sum_{m> \frac{Y}{2}(\log Y)^2} \frac{m\abs{\Lambda_f(m)}^2 e^{-\frac{2m}{Y}}}{m^{2\sigma_0}} &\ll \sum_{m> \frac{Y}{2}(\log Y)^2} \frac{m\abs{\Lambda_f(m)}^2 e^{-\frac{m}{Y}} \frac{Y^2}{m^2}}{m^{2\sigma_0}} \\
& \ll Y^2 \sum_{m > \frac{Y}{2}(\log Y)^2} \frac{\abs{\Lambda_f(m)}^2 e^ {-\frac{m}{Y}}}{m^{1+2\sigma_0}}. \\
\end{align*}
Since $\frac{m}{Y} \geq \frac{1}{2}(\log Y)^2$ for $m \geq \frac{Y}{2}(\log Y)^2$, we have $e^{\frac{m}{Y}} \gg Y^B$ for any large positive constant $B$. Therefore,
\begin{align*}
\sum_{m> \frac{Y}{2}(\log Y)^2} \frac{m\abs{\Lambda_f(m)}^2 e^{-\frac{2m}{Y}}}{m^{2\sigma_0}} &\ll \frac{Y^2}{Y^B} \sum_{m > \frac{Y}{2} (\log Y)^2} \frac{\abs{\Lambda_f(m)}^2}{m^{1+2\sigma_0}} \\
&\ll 1. \\
\end{align*}
\end{proof}

\begin{lemma} \label{l6}
Assuming Rudnick--Sarnak conjecture and taking $Y$ sufficiently large, we have
$$ \sum_{m \leq \frac{Y}{2}(\log Y)^2} \frac{\abs{\Lambda_f(m)}^2}{m^{2 \sigma_0}} e^{-\frac{2m}{Y}} \ll (\log Y)^2. $$ \\
\end{lemma}

\begin{proof}
\noindent
Note that
$$ \sum_{m \leq \frac{Y}{2}(\log Y)^2} \frac{\abs{\Lambda_f(m)}^2}{m^{2 \sigma_0}} e^{-\frac{2m}{Y}} \leq \sum_{p \leq \frac{Y}{2}(\log Y)^2} \frac{(\log p)^2 \abs{a_f(p)}^2}{p^{2 \sigma_0}} + \sum_{r=2}^{ \left[ \frac{\log \frac{Y}{2} }{\log 2} \right] +1} \sum_p \frac{(\log p)^2 \abs{a_f(p^r)}^2}{(p^r)^{2 \sigma_0}}, $$
and
$$ \abs{a_f(p)} = \abs{\sum_{i=1}^n \alpha_{p,i} } = \abs{A(p,1,\dots,1)}.  $$ \\

\noindent
We have,
\begin{align*}
\sum_{m \leq Y} \frac{c_m}{m^l} &= \int_1^Y \frac{d \left( \sum_{m \leq u} c_m \right) }{u^l} \\
&= \frac{ \sum_{m \leq u} c_m }{u^l} \bigg|_1^Y - \int_1^Y (-l) \frac{ \sum_{m \leq u} c_m }{u^{l+1}} du. \\
\end{align*}

\noindent
From Remark 12.1.8 of \cite{Dg}, we have
$$ \sum_{m_1^{n-1}m_2^{n-2} \dots m_{n-1} \leq Y} \abs{A(m_1,m_2,\dots,m_{n-1})}^2 \ll_f Y .$$
Therefore, 
$$ \sum_{m \leq Y} \abs{A(m,1,\dots,1)}^2 \leq \sum_{m_1^{n-1}m_2^{n-2} \dots m_{n-1} \leq Y} \abs{A(m_1,m_2,\dots,m_{n-1})}^2 \ll_f Y .$$ \\

\noindent
Taking $l=2 \sigma_0$ and $c_m = \abs{A(m,1,\dots,1)}^2$,
$$ \sum_{m \leq \frac{Y}{2}(\log Y)^2} \frac{\abs{A(m,1,\dots,1)}^2}{m^{2 \sigma_0}} \ll 1. $$ \\

\noindent
Hence,
\begin{align*}
\sum_{p \leq \frac{Y}{2}(\log Y)^2} \frac{(\log p)^2 \abs{a_f(p)}^2}{p^{2 \sigma_0}} \ll (\log Y)^2 \sum_{m \leq \frac{Y}{2}(\log Y)^2} \frac{\abs{A(m,1,\dots,1)}^2}{m^{2 \sigma_0}} \ll (\log Y)^2. \\
\end{align*}

\noindent
By Rudnick--Sarnak conjecture and the bound $\abs{\alpha_{p,i}} \leq p^{\theta_n} $ with $\theta_n = \frac{1}{2} - \frac{1}{n^2+1}$,
$$ \sum_{r \geq 2} \sum_p \frac{(\log p)^2 \abs{a_f(p^r)}^2}{p^r} $$
converges (as in proof of Theorem \ref{t1}) and in particular,
$$ \sum_{r=2}^{\left[ \frac{\log \frac{Y}{2}}{\log 2} \right] +1} \sum_p \frac{(\log p)^2 \abs{a_f(p^r)}^2}{p^r} \ll 1.$$ \\

\noindent
Therefore,
$$ \sum_{m \leq \frac{Y}{2}(\log Y)^2} \frac{\abs{\Lambda_f(m)}^2}{m^{2 \sigma_0}} \ll (\log Y)^2. $$ \\
\end{proof}

\section{Proof of Theorem \ref{t1}}

Assuming $\abs{\alpha_{p,i}} \leq p^{\theta_n}$ with $\theta_n \leq \frac{1}{4}-\epsilon_1$, we need to prove that for every integer $n \geq 5$ and for every integer $r \geq 2$, 
$$ \sum_p \frac{(\log p)^2 \abs{a_f(p^r)}^2}{p^r} < \infty . $$ 
It is enough to show that
$$ \sum_{r=2}^{\infty} \sum_p \frac{(\log p)^2 \abs{a_f(p^r)}^2}{p^r} < \infty .$$ \\

\noindent
Using
$$ a_f(p^r) := \sum_{i=1}^n \alpha_{p,i}^r \qquad \text{and} \qquad \abs{\alpha_{p,i}} \leq p^{\theta_n}$$
we get,
\begin{align*}
\sum_{r=2}^{\infty} \sum_p \frac{(\log p)^2 \abs{a_f(p^r)}^2}{p^r} &\leq \sum_{r=2}^{\infty} \sum_p \frac{(\log p)^2 \left( \sum \limits_{i=1}^n p^{r \theta_n} \right)^2}{p^r} \\
&= \sum_{r=2}^{\infty} \sum_p \frac{(\log p)^2 n^2 p^{2r \theta_n}}{p^r} \\
& \leq n^2 \sum_p (\log p)^2 \sum_{r=2}^{\infty} \frac{p^{2r \left( \frac{1}{4}-\epsilon_1 \right)}}{p^r} \\
&= n^2 \sum_p (\log p)^2 \sum_{r=2}^{\infty} \frac{1}{p^{\frac{r}{2} +2r \epsilon_1}} \\
&= n^2 \sum_p (\log p)^2 \frac{p^{-(1+4\epsilon_1)}}{1-p^{-(\frac{1}{2}+2\epsilon_1)}} \\
&= n^2 \sum_p (\log p)^2 \frac{1}{p^{\frac{1}{2}+2\epsilon_1} \left( p^{\frac{1}{2}+2\epsilon_1}-1 \right)} \\
&\ll_{n,\epsilon_1} 1. 
\end{align*}
This proves Theorem \ref{t1}. \\

\section{Proof of Theorem \ref{t2}}

\noindent
First, we wish to approximate $\frac{L_f'}{L_f}(s)$ uniformly for $\frac{1}{2} < \sigma_0 \leq \sigma \leq \sigma_1 <1$ when $T \leq t \leq 2T$. We assume throughout below the Riemann hypothesis for $L_f(s)$. \\

\noindent
From the work of Godement--Jacquet \cite{RgHj}, it is known that the function $L_f(s)$ is of finite order in any bounded vertical strip. Hence, we can very well assume that
$$ L_f(s) \ll T^A = e^{A \log T} $$
for $-1 \leq \sigma \leq 2$, $T \leq t \leq 2T$ and $A$ some fixed positive constant. \\

\noindent
Taking $s_0=2+it$ with $t \in \mathbb{R}$, we have
$$ L_f(2+it) = \prod_p \prod_{i=1}^n \left( 1-\frac{\alpha_{p,i}}{p^{2+it}} \right)^{-1} .$$ \\

\noindent
Observe that
\begin{align*}
\abs{1-\frac{\alpha_{p,i}}{p^{2+it}}} &\leq 1+ \frac{\abs{\alpha_{p,i}}}{p^2} \\
&\leq 1+ \frac{p^{\theta_n}}{p^2} \\
&= 1+ \frac{1}{p^{2-\theta_n}} \\
& \leq 1+\frac{1}{p^{\frac{3}{2}}} 
\end{align*}
because $\theta_n \leq \frac{1}{2}$ for $n \geq 2$. \\

\noindent
Therefore,
\begin{align*}
\abs{L_f(2+it)} &\geq \prod_p \prod_{i=1}^n \left( 1+\frac{1}{p^{\frac{3}{2}}} \right)^{-1} \\
&= \prod_p \left( 1+\frac{1}{p^{\frac{3}{2}}} \right)^{-n} \\
&= \prod_p \left( \frac{1-\frac{1}{p^{\frac{3}{2}}}}{1-\frac{1}{p^3}} \right)^n \\
&= \left( \frac{\zeta(3)}{\zeta \left( \frac{3}{2} \right)} \right)^n
\end{align*}
which is a constant depending only on $n$. Therefore, $L_f(2+it) \neq 0 \; \; \forall \; \; t \in \mathbb{R}$. \\

\noindent
Hence from Lemma \ref{l1}, with $r=12$, $s_0=2+iT$, $f(s)=L_f(s)$, $M=A \log T$, we obtain
$$ -\frac{L_f'}{L_f}(s) = \sum_{\abs{s-s_0} \leq 6} \frac{1}{s-\rho} + O(\log T). $$ \\
For $\abs{s-s_0} \leq 3$ and so in particular for $-1 \leq \sigma \leq 2, t=T$, replacing T by t in the particular case, we obtain
$$ -\frac{L_f'}{L_f}(s) = \sum_{\abs{\rho-s_0} \leq 6} \frac{1}{s-\rho} + O(\log t). $$ \\

\noindent
Any term occuring in $\sum \limits_{\abs{t-\gamma} \leq 1} \frac{1}{s-\rho}$ but not in $\sum \limits_{\abs{s-s_0} \leq 6} \frac{1}{s-\rho}$ is bounded and the number of such terms does not exceed
$$ N_f^*(t+6)-N_f^*(t-6) \ll \log t ,$$
where $N_f^*(t)$ is the number of zeros of $L_f(s)$ in the region $0 \leq \sigma \leq 1$ and $0 \leq t \leq T$. Thus, we get
$$ -\frac{L_f'}{L_f}(s) = \sum_{\abs{t-\gamma} \leq 1} \frac{1}{s-\rho} +O(\log t). $$ \\

\noindent
Assume $\frac{1}{2} < \sigma <1$ and $T \leq t \leq 2T$, then
$$ \sum_{m=1}^{\infty} \frac{\Lambda_f(m)}{m^s}e^{-\frac{m}{Y}} = - \frac{1}{2 \pi i} \int_{2-i\infty}^{2+i\infty} \frac{L_f'}{L_f}(s+w) \Gamma(w) Y^w dw .$$
Note also that from the above reasoning
$$ \frac{L_f'}{L_f}(s) \ll \log t \qquad \text{on any line} \; \sigma \neq \frac{1}{2}.$$
Also,
$$ \frac{L_f'}{L_f}(s) \ll \frac{\log t}{\min (\, \abs{t-\gamma})} + \log t \qquad \text{uniformly for} \; -1 \leq \sigma \leq 2.$$ \\

\noindent
From Lemma \ref{l2}, we observe that each interval $(j,j+1)$ contains values of $t$ whose distance from the ordinate of any zero exceeds $\frac{A}{\log j}$, there is a $t_j$ in any such interval for which 
$$ \frac{L_f'}{L_f}(s) \ll (\log t)^2 \qquad \text{where} \; -1 \leq \sigma \leq 2 \; \text{and} \; t=t_j.$$ \\

\noindent
Applying Cauchy's residue theorem to the rectangle, we get
\begin{center}
\begin{tikzpicture}

\draw (0,3) -- (0,0) node[pos=0.5]{\tikz \draw[- angle 90] (0,1 pt) -- (0,-1 pt);};
\draw (0,0) -- (5,0) node[pos=0.5]{\tikz \draw[- angle 90] (-1 pt,0) -- (1 pt,0);};
\draw (5,0) -- (5,3) node[pos=0.5]{\tikz \draw[- angle 90] (0,-1 pt) -- (0,1 pt);};
\draw (5,3) -- (0,3) node[pos=0.5]{\tikz \draw[- angle 90] (1 pt,0) -- (-1 pt,0);};

\node at (0,0) {$\boldsymbol{\cdot}$};
\node at (0,3) {$\boldsymbol{\cdot}$};
\node at (5,0) {$\boldsymbol{\cdot}$};
\node at (5,3) {$\boldsymbol{\cdot}$};

\coordinate [label= left:$\frac{1}{4}-\sigma-i t_j$] (A) at (0,0);
\coordinate [label= left:$\frac{1}{4}-\sigma+i t_j$] (D) at (0,3);
\coordinate [label= right:$2-i t_j$] (B) at (5,0);
\coordinate [label= right:$2+i t_j$] (C) at (5,3);

\end{tikzpicture}
\end{center}

\begin{align*}
&\frac{1}{2 \pi i} \left( \int_{2-i t_j}^{2+i t_j} + \int_{2+i t_j}^{\frac{1}{4}-\sigma+i t_j} + \int_{\frac{1}{4}-\sigma+i t_j}^{\frac{1}{4}-\sigma-i t_j} + \int_{\frac{1}{4}-\sigma-i t_j}^{2-i t_j} \right) \frac{L_f'}{L_f}(s+w) \Gamma(w) Y^w dw \\
&= \frac{L_f'}{L_f}(s) + \sum_{-t_j < \gamma < t_j} \Gamma(\rho-s) Y^{\rho-s}. \\
\end{align*}

\noindent
In the sum appearing on the right hand side above, zeros $\rho$ are counted with its multiplicity if there are any multiple zeros. The integrals along the horizontal lines tend to zero as $j \to \infty$ since gamma function decays exponentially and $Y$ is going to be at most a power of $T$ only, so that
$$ \sum_{m=1}^{\infty} \frac{\Lambda_f(m)}{m^s} e^{-\frac{m}{Y}} = \frac{1}{2 \pi i} \int_{\frac{1}{4}-\sigma -i \infty}^{\frac{1}{4}-\sigma +i \infty} \frac{L_f'}{L_f}(s+w) \Gamma(w) Y^w dw - \frac{L_f'}{L_f}(s) - \sum_{\rho} \Gamma(\rho-s) Y^{\rho -s} .$$ \\

\noindent
Note that $ \Gamma(w) \ll e^{-A \, \abs{v}} $ so that the integral on $\Re(w) = \frac{1}{4} - \sigma$ is
\begin{align*}
&\ll \int_{- \infty}^{\infty} e^{-A \, \abs{v}} \log (\, \abs{t+v}+2) Y^{\frac{1}{4}-\sigma} dv \\
&\ll \int_0^{2t} e^{-A \, \abs{v}} \log (10 \abs{t} +2) Y^{\frac{1}{4}-\sigma} dv + \left( \int_{-\infty}^0 + \int_{2t}^{\infty} \right) e^{-A \, \abs{v}} \log (\, \abs{v}+10) Y^{\frac{1}{4}-\sigma} dv \\
&\ll Y^{\frac{1}{4}-\sigma} \log T + Y^{\frac{1}{4}-\sigma} \\
&\ll Y^{\frac{1}{4}-\sigma} \log T. \\
\end{align*}

\noindent
Note that for $\frac{1}{2} < \sigma_0 \leq \sigma \leq \sigma_1 <1$,
$$ \abs{\Gamma(\rho-s)} < A_1 e^{-A_2 \, \abs{\gamma-t}} $$
uniformly for $\sigma$ in the said range.
Therefore,
$$ \sum_{\rho} \abs{\Gamma(\rho-s)} < A_1 \sum_{\rho} e^{-A_2 \, \abs{\gamma-t}} = A_1 \sum_{m=1}^{\infty} \sum_{m-1 \leq \gamma \leq m} e^{-A_2 \, \abs{t-\gamma}} .$$ \\

\noindent
The number of terms in the inner sum is 
$$ \ll \log (\, \abs{t}+m) \ll \log \, \abs{t}+\log (m+1)$$ 
and hence
$$\sum_{\rho} \abs{\Gamma(\rho-s)} \ll \sum_{m=1}^{\infty} e^{-A_2 m} (\log \, \abs{t} + \log (m+1)) \ll \log T ,$$
$$\abs{\sum_{\rho} \Gamma(\rho-s) Y^{\rho-s}} \ll Y^{\frac{1}{2}-\sigma} \log T .$$ \\
Thus for $\frac{1}{2} < \sigma_0 \leq \sigma \leq \sigma_1 <1$, we have
$$ - \frac{L_f'}{L_f}(s) = \sum_{m=1}^{\infty} \frac{\Lambda_f(m)}{m^s} e^{-\frac{m}{Y}} + O_f(Y^{\frac{1}{2} - \sigma} \log T) .$$ \\

\noindent
Thus for $\frac{1}{2}+\epsilon \leq \sigma_0 \leq 1-\epsilon$ and $T \leq t \leq 2T$, we obtain

\begin{align*}
\abs{ \frac{L_f'}{L_f}(\sigma_0 + it) }^2 &\ll \abs{ \sum_{m=1}^{\infty} \frac{\Lambda_f(m)e^{-\frac{m}{Y}}}{m^{\sigma_0+it}} }^2 + \left( Y^{\frac{1}{2}-\sigma_0} \log T \right)^2 .\\
\end{align*}
Thus,
\begin{align*}
\int_T^{2T} \abs{ \frac{L_f'}{L_f}(\sigma_0 + it) }^2 dt &\ll_f \int_T^{2T} \abs{ \sum_{m=1}^{\infty} \frac{\Lambda_f(m)e^{-\frac{m}{Y}}}{m^{\sigma_0+it}} }^2 dt + Y^{1-2\sigma_0} T (\log T)^2 . \\
\end{align*}

\noindent
We note that
$$ \abs{ \sum_{m=1}^{\infty} \frac{\Lambda_f(m)e^{-\frac{m}{Y}}}{m^{\sigma_0+it}} }^2 \ll \abs{ \sum_{m \leq \frac{Y}{2} (\log Y)^2} \frac{\Lambda_f(m)e^{-\frac{m}{Y}}}{m^{\sigma_0+it}} }^2 + \abs{ \sum_{m> \frac{Y}{2}(\log Y)^2} \frac{\Lambda_f(m)e^{-\frac{m}{Y}}}{m^{\sigma_0+it}} }^2,  $$
and hence
\begin{align*}
\int_T^{2T} \abs{ \frac{L_f'}{L_f}(\sigma_0 + it) }^2 dt &\ll_f \int_T^{2T} \abs{ \sum_{m \leq \frac{Y}{2} (\log Y)^2} \frac{\Lambda_f(m)e^{-\frac{m}{Y}}}{m^{\sigma_0+it}} }^2 + \int_T^{2T} \abs{ \sum_{m> \frac{Y}{2}(\log Y)^2} \frac{\Lambda_f(m)e^{-\frac{m}{Y}}}{m^{\sigma_0+it}} }^2 \\
& \quad + Y^{1-2\sigma_0} T (\log T)^2 . \\
\end{align*}

\noindent
By Montgomery--Vaughan theorem (Lemmas \ref{l3} and \ref{l4}) and Lemma \ref{l5}, we get
\begin{align*}
\int_T^{2T} \abs{ \frac{L_f'}{L_f}(\sigma_0 + it) }^2 dt & \ll_f \sum_{m \leq \frac{Y}{2} (\log Y)^2 } \frac{\abs{\Lambda_f(m)}^2e^{-\frac{2m}{Y}}}{m^{2 \sigma_0}} \left( T+O(m) \right)  \\
& \quad+ \sum_{m > \frac{Y}{2}(\log Y)^2} \frac{\abs{\Lambda_f(m)}^2e^{-\frac{2m}{Y}}}{m^{2 \sigma_0}} \left( T+O(m) \right) + Y^{1-2\sigma_0} T (\log T)^2 \\ 
& \ll_f T \sum_{m \leq \frac{Y}{2} (\log Y)^2 } \frac{\abs{\Lambda_f(m)}^2e^{-\frac{2m}{Y}}}{m^{2 \sigma_0}} + \sum_{m \leq \frac{Y}{2} (\log Y)^2 } m \frac{\abs{\Lambda_f(m)}^2e^{-\frac{2m}{Y}}}{m^{2 \sigma_0}} \\
& \quad + T \sum_{m > \frac{Y}{2}(\log Y)^2} \frac{\abs{\Lambda_f(m)}^2e^{-\frac{2m}{Y}}}{m^{2 \sigma_0}} + \sum_{m > \frac{Y}{2}(\log Y)^2} m \frac{\abs{\Lambda_f(m)}^2e^{-\frac{2m}{Y}}}{m^{2 \sigma_0}} \\
& \quad + Y^{1-2\sigma_0} T (\log T)^2. \\
\end{align*}

\noindent
By Lemmas \ref{l5} and \ref{l6}, we obtain
\begin{align*}
\int_T^{2T} \abs{\frac{L_f'}{L_f} \left( \frac{1}{2} + \epsilon + it \right) }^2 dt \ll_{f,n,\epsilon} T(\log Y)^2 + Y(\log Y)^4 + Y^{1-2\sigma_0} T (\log T)^2 .\\
\end{align*}
We choose $Y = \exp \{(\log T)^{\eta} \}$ with any $\eta$ satisfying $0<\eta<\frac{1}{2}$ so that we obtain
\begin{align*}
\int_T^{2T} \abs{\frac{L_f'}{L_f} \left( \sigma_0 + it \right) }^2 dt \ll_{f,n,\epsilon,\eta} T(\log T)^{2 \eta} .\\
\end{align*}
This proves Theorem \ref{t2}. 

\section*{Acknowledgements}
First author is thankful to UGC for its supporting NET Junior Research Fellowship with UGC Ref. No. : 1004/(CSIR--UGC NET Dec. 2017). The authors are thankful to the anonymous referee for some fruitful comments. \\

\bibliographystyle{amsalpha}

\begin{thebibliography}{6}

\bibliographysize{small}
  
\bibitem{RgHj}
R. Godement, \& H. Jacquet, \emph{Zeta functions of simple algebras} (Vol. 260). Springer (2006). \url{https://doi.org/10.1007/BFb0070263}

\bibitem{Dg}
D. Goldfeld, \emph{Automorphic forms and L-functions for the group GL(n, R)} (Vol. 99) . Cambridge University Press (2006). \url{https://doi.org/10.1017/CBO9780511542923}

  
\bibitem{YjGl}
Y. Jiang, \&  G. L\"{u}, \emph{Exponential sums formed with the von Mangoldt function and Fourier coefficients of ${GL}(m)$ automorphic forms}. Monatshefte für Mathematik, 184(4), 539-561  (2017). \url{https://doi.org/10.1007/s00605-017-1068-4}

\bibitem{Hk}
H. Kim,  \emph{Functoriality for the exterior square of $GL_4$ and the symmetric fourth of $GL_2$}. Journal of the American Mathematical Society, 16(1), 139-183 (2003). \url{https://doi.org/10.1090/S0894-0347-02-00410-1}

\bibitem{Hh}
H. H. Kim,  \emph{A note on Fourier coefficients of cusp forms on $GL_n$}. Forum Math. 18, 115-119 (2006). \url{https://doi.org/10.1515/FORUM.2006.007}

\bibitem{Kr}
K. Ramachandra, \emph{Some remarks on a theorem of Montgomery and Vaughan}. Journal of Number Theory, 11(3), 465-471 (1979). \url{https://doi.org/10.1016/0022-314X(79)90011-8}

\bibitem{KrAs}
K. Ramachandra, \& A. Sankaranarayanan, \emph{Notes on the Riemann zeta-function}. The Journal of the Indian Mathematical Society, 57(1-4), 67-77 (1991). \url{http://informaticsjournals.in/index.php/jims/article/view/21900}

\bibitem{ZrPs}
Z. Rudnick, \& P. Sarnak,  \emph{Zeros of principal L-functions and random matrix theory}. Duke Mathematical Journal, 81(2), 269-322 (1996). \url{https://doi.org/10.1215/S0012-7094-96-08115-6}

\bibitem{EtHb}
E. C. Titchmarsh, \& D. R. Heath-Brown,  \emph{The theory of the Riemann zeta-function}. Oxford university press (1986). \url{https://doi.org/10.1112/blms/20.1.77}

\bibitem{HlRc}
H.L. Montgomery, \& R.C. Vaughan,  \emph{Hilbert's inequality}. J. London Math. Soc., 2(8), 73-82 (1974). \url{https://doi.org/10.1112/jlms/s2-8.1.73}

  
\end{thebibliography}

\authoraddresses{Amrinder Kaur \\
               School of Mathematics and Statistics \\
               University of Hyderabad \\
               Hyderabad - 500046 \\
               Telangana \\
               India. \\
              \email{amrinder1kaur@gmail.com}           \\

              Ayyadurai Sankaranarayanan  \\
              School of Mathematics and Statistics \\
              University of Hyderabad \\
               Hyderabad - 500046 \\
               Telangana \\
               India. \\
              \email{sank@uohyd.ac.in}
}

\end{document}